\documentclass[12pt]{article}

\usepackage[T1]{fontenc}
\usepackage{lmodern}
\usepackage{amsmath,amssymb,amsthm,mathtools}
\usepackage{aliascnt}
\usepackage{microtype}
\usepackage[hidelinks]{hyperref}
\usepackage{enumitem}

\numberwithin{equation}{section}

\newtheorem{theorem}{Theorem}[section]
\newaliascnt{lemma}{theorem}
\newtheorem{lemma}[lemma]{Lemma}
\aliascntresetthe{lemma}
\newaliascnt{corollary}{theorem}
\newtheorem{corollary}[corollary]{Corollary}
\aliascntresetthe{corollary}
\newaliascnt{proposition}{theorem}
\newtheorem{proposition}[proposition]{Proposition}
\aliascntresetthe{proposition}
\newaliascnt{problem}{theorem}
\newtheorem{problem}[problem]{Problem}
\aliascntresetthe{problem}
\theoremstyle{remark}
\newaliascnt{remark}{theorem}
\newtheorem{remark}[remark]{Remark}
\aliascntresetthe{remark}
\usepackage[nameinlink,capitalize,noabbrev]{cleveref}

\newcommand{\one}{\mathbf{1}}
\newcommand{\tr}{\operatorname{tr}}

\newcommand{\ip}[2]{\left\langle #1,#2\right\rangle}
\newcommand{\splus}{s^{+}}
\newcommand{\sminus}{s^{-}}

\newenvironment{pf}[1][Proof]{\noindent\textbf{#1.} }{\hfill\rule{1mm}{2mm} \par\vspace{0.6\baselineskip}}

\begin{document}

\title{From the Square-Energy Conjecture to Signed Graphs: Sharp Bounds for Positive Square Energy}
\author{Fu-Tao Hu\thanks{E-mail address: hufu@ahu.edu.cn}, Xiao Han \\
{\small Center for Pure Mathematics, School of Mathematical Sciences, Anhui University,}\\
{\small Hefei, 230601, P.R. China}
}

\date{}
\maketitle

\begin{abstract}
Let $\Sigma=(G,\sigma)$ be a connected signed graph of order $n$ and size $m$, and let $s^{+}(\Sigma)$ and $s^{-}(\Sigma)$ denote the sums of the squares of its positive and negative adjacency eigenvalues, respectively.  The square-energy conjecture of Elphick, Farber, Goldberg, and Wocjan states that every connected graph $G$ of order $n$ satisfies
\[
  \min\{s^{+}(G),s^{-}(G)\}\ge n-1.
\]
Liu and Ning~\cite{LiuNing2023} published a wide-ranging paper entitled ``Unsolved Problems in spectral graph theory", and
this conjectures were placed first in their list of such problems.
We prove that every signature $\sigma$ of a connected graph $G$ satisfies the sharp bound
\[
  s^{+}(\Sigma)\le 2m-n+1.
\]
For the all-positive signing this gives $s^{+}(G)\le 2m-n+1$, whereas for the all-negative signing it gives $s^{-}(G)\le 2m-n+1$.  Since $s^{+}(G)+s^{-}(G)=2m$, these two special cases imply the square-energy conjecture; the present theorem is stronger in scope because the same bound holds for every signing of $G$.  Applying the theorem to the negation $-\Sigma$ also yields
\[
  s^{+}(\Sigma)\ge n-1.
\]
Both bounds are sharp.  The proof is based on a doubly nonnegative matrix inequality.  We also shorten the proof of that inequality by replacing its final case distinction with a fixed convex combination.

 
\end{abstract}

\section{Introduction}

A signed graph is a pair $\Sigma=(G,\sigma)$, where $G$ is a graph and the signature
\[
  \sigma:E(G)\longrightarrow\{+1,-1\}
\]
assigns a sign to each edge.  The subject goes back to Harary's theory of structural balance~\cite{Harary1953} and was developed systematically by Zaslavsky~\cite{Zaslavsky1982}.  Switching, balance, and the interaction between the signature and the underlying graph are central throughout the theory.  For spectral aspects, we refer to the survey of Belardo, Cioab\u{a}, Koolen, and Wang~\cite{BelardoEtAl2018} and the monograph of Stani\'c~\cite{Stanic2026}.

The adjacency matrix $A(\Sigma)$ is the symmetric $\{0,\pm1\}$-matrix whose $uv$-entry is $\sigma(uv)$ when $uv\in E(G)$ and is $0$ otherwise.  Switching conjugates $A(\Sigma)$ by a diagonal $\{\pm1\}$-matrix and therefore preserves the adjacency spectrum.  If $G$ is connected, comparison with the adjacency matrix of the underlying graph gives
\[
 \lambda_1(\Sigma)\le \lambda_1(G),
 \qquad
 \lambda_n(\Sigma)\ge-\lambda_1(G).
\]
Equality in the first inequality characterizes balanced signatures, while equality in the second characterizes antibalanced signatures.  Consequently, the adjacency spectral radius of $\Sigma$ is at most that of $G$, with equality precisely in the balanced or antibalanced cases; see~\cite{BelardoEtAl2018,Stanic2026}.  In contrast with unsigned graphs, $A(\Sigma)$ need not be entrywise nonnegative, so Perron--Frobenius arguments are generally unavailable; this is one of the basic difficulties in the spectral theory of signed graphs.  The classical energy, the sum of the absolute values of all adjacency eigenvalues, has also been studied for signed graphs; see, for example, Germina, Hameed, and Zaslavsky~\cite{GerminaHameedZaslavsky2011}.

For an unsigned graph $G$, Wocjan and Elphick~\cite{WocjanElphick2013} introduced the positive and negative square energies
\[
 \splus(G):=\sum_{\lambda_i(G)>0}\lambda_i(G)^2,
 \qquad
 \sminus(G):=\sum_{\lambda_i(G)<0}\lambda_i(G)^2
\]
in connection with spectral lower bounds for the chromatic number.  Ando and Lin~\cite{AndoLin2015} proved the corresponding chromatic inequality.  Since
\begin{equation}\label{eq:unsigned-trace-intro}
  \splus(G)+\sminus(G)=\tr A(G)^2=2m,
\end{equation}
these two quantities divide the total adjacency square energy according to the signs of the eigenvalues.

Elphick, Farber, Goldberg, and Wocjan~\cite{ElphickFarberGoldbergWocjan2016} initiated the systematic study of square energies and proposed the following conjecture: every connected graph of order $n$ satisfies
\begin{equation}\label{eq:unsigned-conjecture}
  \min\{\splus(G),\sminus(G)\}\ge n-1.
\end{equation}
By~\eqref{eq:unsigned-trace-intro}, this is equivalent to the pair of upper bounds
\begin{equation}\label{eq:unsigned-equivalent-upper}
  \splus(G)\le 2m-n+1
  \qquad\text{and}\qquad
  \sminus(G)\le 2m-n+1.
\end{equation}
The conjecture was highlighted in the survey of Liu and Ning~\cite{LiuNing2023}.  It was verified for several graph classes in the original paper and in subsequent work of Abiad, de Lima, Desai, Guo, Hogben, and Madrid~\cite{AbiadEtAl2023}.  Further developments include Nordhaus--Gaddum inequalities~\cite{ElphickAouchiche2017}, the study of symmetry and asymmetry between the two square energies~\cite{ElphickLinz2024}, semidefinite and conic formulations~\cite{CoutinhoSpier2023,CoutinhoSpierZhang2024,Zhang2024}, linear lower bounds~\cite{AkbariEtAl2025}, and the wider theory of positive and negative $p$-energies~\cite{TangLiuWang2026,ElphickTangZhang2026}.  The conjecture~\eqref{eq:unsigned-conjecture} was recently proved by Liu, Tang, and Zhang~\cite{LiuTangZhang2026}.

The decisive ingredient in~\cite{LiuTangZhang2026} is an inequality for doubly nonnegative matrices indexed by the vertex set of a graph.  The inequality depends only on the underlying graph, whereas the signature enters the spectral argument through one signed sum over the edges.  This separation makes the method well suited to signed graphs.

Our main result extends the upper bounds in~\eqref{eq:unsigned-equivalent-upper} to every signing of the underlying graph.

\begin{theorem}\label{thm:main}
Let $\Sigma=(G,\sigma)$ be a connected signed graph of order $n$ and size $m$.  Then
\begin{equation}\label{eq:main-bound}
  \splus(\Sigma)\le 2m-n+1.
\end{equation}
\end{theorem}

For a graph $G$, denote by $+G$ and $-G$ its all-positive and all-negative signings, respectively.  Then
\[
 A(+G)=A(G),\qquad A(-G)=-A(G),
\]
and hence
\[
 \splus(+G)=\splus(G),\qquad \splus(-G)=\sminus(G).
\]
Thus \cref{thm:main}, applied first to $+G$ and then to $-G$, gives the two inequalities in~\eqref{eq:unsigned-equivalent-upper}.  Consequently, the square-energy conjecture is contained in \cref{thm:main} as the two constant-sign cases.  The theorem is stronger in scope, since it establishes the same sharp edge bound for every signature of $G$.

Since $\splus(\Sigma)\ge \lambda_1(\Sigma)^2$, \cref{thm:main} also extends the corresponding Hong-type edge bound for the largest positive adjacency eigenvalue; compare~\cite{Hong1988}.  The lower bound follows by applying the theorem to the negation of the signed graph.

\begin{corollary}\label{cor:lower-intro}
Let $\Sigma=(G,\sigma)$ be a connected signed graph of order $n$.  Then
\begin{equation}\label{eq:lower-intro}
  \splus(\Sigma)\ge n-1.
\end{equation}
Consequently,
\[
 \min\{\splus(\Sigma):\Sigma\text{ is a connected signed graph of order }n\}=n-1.
\]
\end{corollary}

Both bounds are sharp.  Every signed tree satisfies $\splus(\Sigma)=n-1$ and attains equality in both bounds.  Balanced complete signed graphs attain equality in the upper bound, whereas antibalanced complete signed graphs attain equality in the lower bound.

For completeness, we include the doubly nonnegative inequality and its proof rather than cite it as a black box.  The folding step and the cut-vertex decomposition follow Liu, Tang, and Zhang~\cite{LiuTangZhang2026}.  In the no-cut-vertex case, however, we replace their final two-case argument by the fixed convex combination
\[
 \frac{n-1}{n}\times(\text{vertex-deletion estimate})
 \;+
 \frac1n\times(\text{edge Cauchy--Schwarz estimate}).
\]
The resulting error term is nonpositive solely because a simple graph has at most $\binom n2$ edges.  This gives a shorter self-contained proof and makes the signed reduction transparent.

The paper is organized as follows.  Section~\ref{sec:preliminaries} gives the required notation and preliminary facts on signed graphs and matrices.  Section~\ref{sec:dnn} proves the doubly nonnegative matrix inequality.  Section~\ref{sec:signed-bound} proves \cref{thm:main} and derives \cref{cor:lower-intro}.  Section~\ref{sec:consequences} records the sharpness examples and the disconnected extension.  In Section~\ref{sec:equality-problem}, we formulate the equality-case problem suggested by the two sharp bounds.

\section{Preliminaries}\label{sec:preliminaries}

All graphs are finite, simple, and undirected.  Let $G=(V,E)$ be a graph with order $n=|V|$ and size $m=|E|$.  A \emph{signed graph} is a pair
\[
  \Sigma=(G,\sigma),
\]
where $\sigma:E(G)\to\{+1,-1\}$ is its \emph{signature}.  The graph $G$ is the \emph{underlying graph} of $\Sigma$, and the order and size of $\Sigma$ are the order and size of $G$.  The signed graph $\Sigma$ is connected if its underlying graph is connected.  We write $+G$ and $-G$ for the all-positive and all-negative signings of $G$, respectively.

The \emph{signed adjacency matrix} of $\Sigma$ is the real symmetric matrix $A(\Sigma)=(a_{uv})_{u,v\in V}$ given by
\[
 a_{uv}=
 \begin{cases}
  \sigma(uv),&uv\in E(G),\\
  0,&uv\notin E(G).
 \end{cases}
\]
In particular, $a_{uu}=0$ for all $u\in V$.  We write
\[
  \lambda_1(\Sigma)\ge\lambda_2(\Sigma)\ge\cdots\ge\lambda_n(\Sigma)
\]
for the eigenvalues of $A(\Sigma)$, counted with multiplicity, and define
\begin{equation}\label{eq:def-square-energies}
 \splus(\Sigma):=\sum_{\lambda_i(\Sigma)>0}\lambda_i(\Sigma)^2,
 \qquad
 \sminus(\Sigma):=\sum_{\lambda_i(\Sigma)<0}\lambda_i(\Sigma)^2.
\end{equation}

For a cycle $C$ of $G$, set $\sigma(C):=\prod_{e\in E(C)}\sigma(e)$.  The signed graph is \emph{balanced} if every cycle has positive sign.  Given a switching function $\theta:V\to\{+1,-1\}$, define
\[
  \sigma^\theta(uv):=\theta(u)\sigma(uv)\theta(v).
\]
If $D_\theta:=\operatorname{diag}(\theta(v):v\in V)$, then
\begin{equation}\label{eq:switching}
 A(G,\sigma^\theta)=D_\theta A(G,\sigma)D_\theta.
\end{equation}
Thus switching preserves the adjacency spectrum and both square energies.  Harary's balance theorem states that a signed graph is balanced if and only if it is switching equivalent to the all-positive signature on its underlying graph~\cite{Harary1953}.  The signed graph is \emph{antibalanced} if it is switching equivalent to the all-negative signature, equivalently, if its negation is balanced.

The \emph{negation} of $\Sigma$ is
\[
  -\Sigma:=(G,-\sigma).
\]
It satisfies
\begin{equation}\label{eq:global-negation}
  A(-\Sigma)=-A(\Sigma).
\end{equation}
Therefore
\begin{equation}\label{eq:negation-exchange}
  \splus(-\Sigma)=\sminus(\Sigma),
  \qquad
  \sminus(-\Sigma)=\splus(\Sigma).
\end{equation}

Let $A=A(\Sigma)$.  Its positive and negative spectral parts are
\[
 A_+:=\sum_{\lambda_i>0}\lambda_i x_ix_i^\top,
 \qquad
 A_-:=-\sum_{\lambda_i<0}\lambda_i x_ix_i^\top,
\]
where $x_1,\ldots,x_n$ is an orthonormal eigenbasis of $A$.  Then
\begin{equation}\label{eq:spectral-parts}
 A=A_+-A_-,
 \qquad A_+,A_-\succeq0,
 \qquad A_+A_-=0,
\end{equation}
and
\begin{equation}\label{eq:energy-parts}
 \splus(\Sigma)=\tr(A_+^2)=\|A_+\|_F^2,
 \qquad
 \sminus(\Sigma)=\tr(A_-^2)=\|A_-\|_F^2.
\end{equation}

For real symmetric matrices $X,Y$, write
\[
 \ip{X}{Y}:=\tr(XY),
 \qquad
 \|X\|_F:=\sqrt{\ip{X}{X}},
\]
and denote their Hadamard product by $X\circ Y$.  A real symmetric matrix $M$ is \emph{doubly nonnegative} if
\[
  M\succeq0
  \qquad\text{and}\qquad
  M\ge0
\]
entrywise.  We use the Schur product theorem: if $X,Y\succeq0$, then $X\circ Y\succeq0$; see~\cite[Section~7.5]{HornJohnson2013}.  For background on doubly nonnegative and completely positive matrices, see~\cite{BermanShakedMonderer2003}.

The following trace identity is fundamental.

\begin{lemma}\label{lem:trace-identity}
If $\Sigma=(G,\sigma)$ has size $m$, then
\begin{equation}\label{eq:trace-2m}
 \splus(\Sigma)+\sminus(\Sigma)=\tr A(\Sigma)^2=2m.
\end{equation}
Consequently,
\begin{equation}\label{eq:positive-complementarity}
 \splus(\Sigma)+\splus(-\Sigma)=2m.
\end{equation}
\end{lemma}

\begin{pf}
Every edge $uv$ contributes the two entries $a_{uv}=a_{vu}=\sigma(uv)$, both of square $1$, and all other off-diagonal entries vanish.  Hence
\[
 \tr A(\Sigma)^2=\sum_{u,v\in V}a_{uv}^2=2m.
\]
The first identity now follows by summing the squares of the eigenvalues, and the second follows from~\eqref{eq:negation-exchange}.
\end{pf}

For a connected graph $G$, define
\begin{equation}\label{eq:q-def}
  q(G):=2|E(G)|-|V(G)|+1.
\end{equation}
We shall prove that $q(G)$ controls a weighted edge sum for every doubly nonnegative matrix indexed by $V(G)$.

\section{A doubly nonnegative edge-sum inequality}\label{sec:dnn}

The next theorem was introduced by Liu, Tang, and Zhang~\cite[Theorem~2.1]{LiuTangZhang2026}.  We give a self-contained proof with a shorter final step.

\begin{theorem}[DNN edge-sum inequality]\label{thm:dnn}
Let $G$ be a connected simple graph, and let $M$ be a doubly nonnegative matrix indexed by $V(G)$.  Then
\begin{equation}\label{eq:dnn-main}
 4\left(\sum_{uv\in E(G)}\sqrt{M_{uv}}\right)^2
 \le q(G)\,\one^\top M\one,
\end{equation}
where each edge is counted once.
\end{theorem}

\begin{pf}
For a pair $(G,M)$ as in the statement, write
\[
 S(G,M):=\sum_{uv\in E(G)}\sqrt{M_{uv}},
 \qquad
 T(M):=\one^\top M\one.
\]
We prove
\begin{equation}\label{eq:dnn-target}
  4S(G,M)^2\le q(G)T(M)
\end{equation}
by induction on $n:=|V(G)|$.

If $n=1$, then $S(G,M)=0$ and $q(G)=0$.  If $n=2$, then $G=K_2$, $q(G)=1$, and
\[
 M=\begin{pmatrix}a&c\\ c&b\end{pmatrix}
\]
with $a,b,c\ge0$.  Positive semidefiniteness gives $c\le\sqrt{ab}\le(a+b)/2$, and therefore
\[
 4S(G,M)^2=4c\le a+b+2c=T(M).
\]
Assume henceforth that $n\ge3$ and that the theorem holds for all connected graphs of order at most $n-1$.

\medskip
\noindent\emph{Folding non-edges.}
For each unordered non-edge $\{u,v\}$, let $e_u,e_v$ be the corresponding standard basis vectors, and define
\begin{equation}\label{eq:folding}
 N:=M+\sum_{\substack{\{u,v\}\subseteq V(G)\\uv\notin E(G)}}
 M_{uv}(e_u-e_v)(e_u-e_v)^\top.
\end{equation}
Each added summand is positive semidefinite.  It cancels the two off-diagonal entries $M_{uv}$ and $M_{vu}$ and adds their contribution to the diagonal.  Hence $N\succeq0$, $N\ge0$ entrywise,
\[
 N_{uv}=0\quad\text{if }u\ne v\text{ and }uv\notin E(G),
 \qquad
 N_{uv}=M_{uv}\quad\text{if }uv\in E(G).
\]
Furthermore,
\[
 \one^\top(e_u-e_v)(e_u-e_v)^\top\one=0,
\]
so $T(N)=T(M)$, while the edge sums of $N$ and $M$ agree.  We may therefore assume, without loss of generality, that
\begin{equation}\label{eq:support}
  M_{uv}=0\qquad\text{whenever }u\ne v\text{ and }uv\notin E(G).
\end{equation}

\medskip
\noindent\emph{Cut-vertex case.}
Suppose that $v$ is a cut vertex of $G$.  Group the connected components of $G-v$ into two nonempty unions with vertex sets $U$ and $W$, and set
\[
 G_1:=G[U\cup\{v\}],
 \qquad
 G_2:=G[W\cup\{v\}].
\]
Then $G_1$ and $G_2$ are connected, their edge sets partition $E(G)$, and
\begin{equation}\label{eq:q-split}
 q(G_1)+q(G_2)=q(G).
\end{equation}

Since $M\succeq0$, it is the Gram matrix of vectors $z_u$ indexed by $V(G)$:
\[
  M_{uw}=\ip{z_u}{z_w}.
\]
No edge joins $U$ and $W$, so~\eqref{eq:support} implies that
\[
 \mathcal U:=\operatorname{span}\{z_u:u\in U\}
 \quad\text{and}\quad
 \mathcal W:=\operatorname{span}\{z_w:w\in W\}
\]
are orthogonal.  Decompose
\[
 z_v=p_U+p_W+p_0,
 \qquad
 p_U\in\mathcal U,
 \quad p_W\in\mathcal W,
 \quad p_0\perp\mathcal U+\mathcal W.
\]
Let $M_1$ be the Gram matrix of the vectors $(z_u)_{u\in U}$ together with $p_U+p_0$ at $v$, and let $M_2$ be the Gram matrix of $(z_w)_{w\in W}$ together with $p_W$ at $v$.  Both matrices are positive semidefinite.  Their off-diagonal entries are either entries of $M$ or, for an entry involving $v$, equal to the corresponding entry of $M$; hence they are also entrywise nonnegative.  The corresponding quantities satisfy
\begin{align}
 S(G,M)&=S(G_1,M_1)+S(G_2,M_2),\label{eq:S-split}\\
 T(M)&=T(M_1)+T(M_2).\label{eq:T-split}
\end{align}
Indeed,~\eqref{eq:T-split} follows by expressing $T(M)$ as the squared norm of the sum of the Gram vectors and using the orthogonal decomposition above.

The induction hypothesis and the scalar Cauchy--Schwarz inequality now give
\begin{align*}
 2S(G,M)
 &\le \sqrt{q(G_1)T(M_1)}+\sqrt{q(G_2)T(M_2)}\\
 &\le \sqrt{\bigl(q(G_1)+q(G_2)\bigr)
                 \bigl(T(M_1)+T(M_2)\bigr)}\\
 &=\sqrt{q(G)T(M)}.
\end{align*}
Thus~\eqref{eq:dnn-target} holds in the cut-vertex case.

\medskip
\noindent\emph{No-cut-vertex case.}
We may now assume that $G-v$ is connected for every $v\in V(G)$.  Put
\[
 S:=S(G,M),
 \qquad T:=T(M),
 \qquad q:=q(G),
\]
and, using~\eqref{eq:support}, write
\begin{equation}\label{eq:T-d-w}
 T=d_0+2w,
 \qquad
 d_0:=\tr M,
 \qquad
 w:=\sum_{uv\in E(G)}M_{uv}.
\end{equation}
This identity uses the support condition~\eqref{eq:support}.

For $v\in V(G)$, define
\[
 r_v:=\sum_{u\sim v}\sqrt{M_{uv}}.
\]
Let $M-v$ denote the principal submatrix obtained by deleting the row and column indexed by $v$.  Then
\[
 S(G-v,M-v)=S-r_v,
 \qquad
 \sum_{v\in V(G)}r_v=2S.
\]
Moreover, if $d(v)$ is the degree of $v$, then
\begin{equation}\label{eq:q-delete}
 q(G-v)=2(m-d(v))-(n-1)+1=q+1-2d(v).
\end{equation}
By induction,
\[
 2(S-r_v)\le\sqrt{q(G-v)T(M-v)}.
\]
Summing over all vertices and applying Cauchy--Schwarz yields
\begin{equation}\label{eq:delete-CS}
 2(n-2)S
 \le \sum_v\sqrt{q(G-v)T(M-v)}
 \le \sqrt{\left(\sum_vq(G-v)\right)
              \left(\sum_vT(M-v)\right)}.
\end{equation}
The two sums are explicit.  Since $\sum_vd(v)=2m$ and $2m=q+n-1$,
\begin{equation}\label{eq:q-delete-sum}
 \sum_vq(G-v)
 =n(q+1)-4m
 =(n-2)(q-1).
\end{equation}
Every diagonal entry of $M$ occurs in $n-1$ of the principal submatrices $M-v$, while each pair of edge entries occurs in $n-2$ of them.  Hence
\begin{equation}\label{eq:T-delete-sum}
 \sum_vT(M-v)
 =(n-1)d_0+2(n-2)w
 =(n-2)T+d_0.
\end{equation}
Substituting~\eqref{eq:q-delete-sum} and~\eqref{eq:T-delete-sum} into~\eqref{eq:delete-CS} and dividing by $(n-2)^2$ gives the averaged estimate
\begin{equation}\label{eq:averaged-estimate}
 4S^2\le(q-1)T+\frac{q-1}{n-2}d_0.
\end{equation}

A second estimate follows by applying Cauchy--Schwarz to the $m$ edge terms:
\[
 S^2\le m\sum_{uv\in E(G)}M_{uv}=mw.
\]
Using~\eqref{eq:T-d-w} and $2m=q+n-1$, we obtain the edge estimate
\begin{equation}\label{eq:edge-estimate}
 4S^2\le4mw=(q+n-1)(T-d_0).
\end{equation}

We now combine~\eqref{eq:averaged-estimate} and~\eqref{eq:edge-estimate}.  Since $4S^2$ is bounded by each right-hand side, it is bounded by their convex combination with weights $(n-1)/n$ and $1/n$:
\begin{align}
 4S^2
 &\le \frac{n-1}{n}
       \left((q-1)T+\frac{q-1}{n-2}d_0\right)
       +\frac1n(q+n-1)(T-d_0)\notag\\
 &=qT+\frac{q-(n-1)^2}{n(n-2)}d_0.
 \label{eq:convex-closure}
\end{align}
Because $G$ is simple,
\[
 q=2m-n+1\le n(n-1)-n+1=(n-1)^2.
\]
Also $d_0=\tr M\ge0$.  Therefore the final term in~\eqref{eq:convex-closure} is nonpositive, and
\[
 4S^2\le qT.
\]
This completes the induction.
\end{pf}

\begin{remark}\label{rem:convex-simplification}
The only change from the proof in~\cite{LiuTangZhang2026} occurs after the two estimates~\eqref{eq:averaged-estimate} and~\eqref{eq:edge-estimate}.  The original proof divides according to the relative sizes of $d_0$ and $w$.  The fixed convex combination~\eqref{eq:convex-closure} removes this case distinction.  Its residual coefficient is exactly
\[
 \frac{q-(n-1)^2}{n(n-2)},
\]
which is nonpositive by the elementary extremal bound $m\le\binom n2$.
\end{remark}

\section{The positive square-energy bound}\label{sec:signed-bound}

We now apply \cref{thm:dnn} to the positive spectral part of a signed adjacency matrix.

\begin{pf}[Proof of \cref{thm:main}]
The case $n=1$ is immediate.  Assume that $n\ge2$, let
\[
  A:=A(\Sigma),
\]
and let $A_+$ be its positive spectral part.  Set
\begin{equation}\label{eq:M-Aplus}
  M:=A_+\circ A_+.
\end{equation}
Since $A_+\succeq0$, the Schur product theorem gives $M\succeq0$, and clearly $M\ge0$ entrywise.  Thus $M$ is doubly nonnegative.

The sum of all entries of $M$ is
\begin{align}
 \one^\top M\one
 &=\sum_{u,v\in V(G)}(A_+)_{uv}^2\notag\\
 &=\|A_+\|_F^2
 =\splus(\Sigma).
 \label{eq:M-entry-sum}
\end{align}
For every edge $uv\in E(G)$,
\[
 \sigma(uv)(A_+)_{uv}\le |(A_+)_{uv}|=\sqrt{M_{uv}}.
\]
Therefore
\begin{align}
 2\sum_{uv\in E(G)}\sqrt{M_{uv}}
 &\ge2\sum_{uv\in E(G)}\sigma(uv)(A_+)_{uv}\notag\\
 &=\ip{A}{A_+}.
 \label{eq:signed-edge-pairing}
\end{align}
The last equality holds because $A$ has zero diagonal and each undirected edge is counted twice in the matrix inner product.  By the spectral definition of $A_+$,
\[
 AA_+=A_+^2,
\]
and hence
\begin{equation}\label{eq:AAplus}
 \ip{A}{A_+}=\tr(AA_+)=\tr(A_+^2)=\splus(\Sigma).
\end{equation}
Combining~\eqref{eq:signed-edge-pairing} and~\eqref{eq:AAplus}, we obtain
\begin{equation}\label{eq:edge-lower-signed}
  2\sum_{uv\in E(G)}\sqrt{M_{uv}}\ge\splus(\Sigma).
\end{equation}

Applying \cref{thm:dnn} to $M$ and using~\eqref{eq:M-entry-sum} and~\eqref{eq:edge-lower-signed}, we get
\[
 \splus(\Sigma)^2
 \le4\left(\sum_{uv\in E(G)}\sqrt{M_{uv}}\right)^2
 \le(2m-n+1)\splus(\Sigma).
\]
It remains to justify division by $\splus(\Sigma)$.  Since $G$ is connected and $n\ge2$, the matrix $A$ is nonzero.  Also $\tr A=0$.  If $A$ had no positive eigenvalue, then all its eigenvalues would be nonpositive and would sum to zero, forcing all of them to vanish, a contradiction.  Thus $\splus(\Sigma)>0$, and division gives
\[
 \splus(\Sigma)\le2m-n+1.
\]
\end{pf}

We derive the lower bound without a separate argument for the negative spectral part.

\begin{pf}[Proof of \cref{cor:lower-intro}]
The negation $-\Sigma=(G,-\sigma)$ has the same connected underlying graph.  By \cref{thm:main},
\[
  \splus(-\Sigma)\le2m-n+1.
\]
Using~\eqref{eq:positive-complementarity},
\[
 \splus(\Sigma)
 =2m-\splus(-\Sigma)
 \ge2m-(2m-n+1)
 =n-1.
\]
The minimum statement follows because every order-$n$ signed tree attains equality; see \cref{prop:trees} below.
\end{pf}

Applying the same results to $-\Sigma$ also gives
\begin{equation}\label{eq:both-signs}
 n-1\le\sminus(\Sigma)\le2m-n+1.
\end{equation}
Thus the corresponding bounds for $\sminus$ follow from the positive bounds by negation.

\section{Sharpness and consequences}\label{sec:consequences}

We first record the tree case.

\begin{proposition}\label{prop:trees}
If $\Sigma=(T,\sigma)$ is a signed tree of order $n$, then
\[
  \splus(\Sigma)=\sminus(\Sigma)=n-1.
\]
Consequently, every signed tree attains equality in both the upper and lower bounds.
\end{proposition}

\begin{pf}
A tree has no cycles, so every signature on $T$ is balanced.  By Harary's theorem, $\Sigma$ is switching equivalent to the all-positive tree $T$.  Since $T$ is bipartite, its adjacency spectrum is symmetric about zero.  Switching preserves the spectrum, and therefore
\[
 \splus(\Sigma)=\sminus(\Sigma).
\]
Because $T$ has $n-1$ edges, \cref{lem:trace-identity} gives
\[
 \splus(\Sigma)+\sminus(\Sigma)=2(n-1),
\]
which proves the result.
\end{pf}

Complete signed graphs provide dense equality examples for both ends of the interval.  The all-positive complete graph has spectrum
\[
  \{n-1,-1^{(n-1)}\},
\]
so
\[
 \splus(+K_n)=(n-1)^2=2\binom n2-n+1.
\]
Thus equality holds in the upper bound.  The all-negative complete graph has spectrum
\[
  \{1^{(n-1)},-(n-1)\},
\]
and hence
\[
 \splus(-K_n)=n-1.
\]
By switching invariance, the upper-bound equality example extends to every balanced complete signed graph, and the lower-bound equality example extends to every antibalanced complete signed graph.

The next proposition characterizes simultaneous equality in the two bounds.

\begin{proposition}\label{prop:simultaneous}
Let $\Sigma=(G,\sigma)$ be connected of order $n$.  Then
\[
 \splus(\Sigma)=\sminus(\Sigma)=n-1
\]
if and only if the underlying graph $G$ is a tree.
\end{proposition}

\begin{pf}
If both square energies equal $n-1$, then \cref{lem:trace-identity} gives $2m=2n-2$, so $m=n-1$.  Since $G$ is connected, it is a tree.  The converse follows from \cref{prop:trees}.
\end{pf}

Finally, additivity over connected components yields the natural disconnected form.

\begin{corollary}\label{cor:disconnected}
Let $\Sigma=(G,\sigma)$ be a signed graph of order $n$ and size $m$, and let $\kappa(G)$ be the number of connected components of $G$.  Then
\begin{equation}\label{eq:disconnected-two-sided}
 n-\kappa(G)
 \le\splus(\Sigma)
 \le2m-n+\kappa(G).
\end{equation}
The same bounds hold for $\sminus(\Sigma)$.
\end{corollary}

\begin{pf}
Let $\Sigma_1,\ldots,\Sigma_{\kappa(G)}$ be the signed connected components, with orders $n_i$ and sizes $m_i$.  The signed adjacency matrix is block diagonal, so $\splus$ is additive over components.  Applying \cref{thm:main,cor:lower-intro} to every component, including one-vertex components, gives
\[
 \splus(\Sigma)
 =\sum_i\splus(\Sigma_i)
 \ge\sum_i(n_i-1)
 =n-\kappa(G)
\]
and
\[
 \splus(\Sigma)
 \le\sum_i(2m_i-n_i+1)
 =2m-n+\kappa(G).
\]
The assertion for $\sminus$ follows by applying the positive statement to $-\Sigma$.
\end{pf}

\begin{remark}\label{rem:signed-unsigned-separation}
The proof keeps the signed and unsigned parts of the argument separate.  The DNN inequality \eqref{eq:dnn-main} depends only on the underlying graph $G$.  The signature appears only in the comparison
\[
 \sigma(uv)(A_+)_{uv}\le |(A_+)_{uv}|.
\]
Thus the edge sum in \cref{thm:dnn} must remain unsigned and nonnegative; inserting $\sigma(uv)$ into that theorem would destroy the doubly nonnegative matrix inequality.  This distinction is essential when passing from graphs to signed graphs.
\end{remark}

\section{An equality-case problem}\label{sec:equality-problem}

The preceding results identify several equality families, but they do not give a complete description of either extremal class.  Since the adjacency spectrum, and hence $\splus$, is invariant under switching, the natural classification is up to switching and isomorphism.

\begin{problem}\label{prob:equality-cases}
Characterize, up to switching and isomorphism, all connected signed graphs $\Sigma=(G,\sigma)$ of order $n$ and size $m$ for which equality holds in at least one of the sharp bounds
\[
  n-1\le \splus(\Sigma)\le 2m-n+1.
\]
Equivalently, determine all connected signed graphs satisfying
\begin{enumerate}[label=\textup{(\roman*)}]
 \item $\splus(\Sigma)=2m-n+1$;
 \item $\splus(\Sigma)=n-1$.
\end{enumerate}
\end{problem}

Every signed tree satisfies both equalities.  Balanced complete signed graphs satisfy~\textup{(i)}, and antibalanced complete signed graphs satisfy~\textup{(ii)}.  The two equality classes intersect exactly in the signed trees: indeed, equality in both bounds gives $2m-n+1=n-1$, and hence $m=n-1$.  Determining all remaining equality cases is an open problem.

\section*{Acknowledgments}

Supported by National Natural Science Foundation of China (No. 12331012).


\begin{thebibliography}{99}

\bibitem{AbiadEtAl2023}
A.~Abiad, L.~de Lima, D.~N.~Desai, K.~Guo, L.~Hogben, and J.~Madrid,
\emph{Positive and negative square energies of graphs},
Electron. J. Linear Algebra \textbf{39} (2023), 307--326.

\bibitem{AkbariEtAl2025}
S.~Akbari, H.~Kumar, B.~Mohar, and S.~Pragada,
\emph{A linear lower bound for the square energy of graphs},
Electron. J. Combin. \textbf{32} (2025), no.~3, Paper No.~P3.53.

\bibitem{AndoLin2015}
T.~Ando and M.~Lin,
\emph{Proof of a conjectured lower bound on the chromatic number of a graph},
Linear Algebra Appl. \textbf{485} (2015), 480--484.

\bibitem{BelardoEtAl2018}
F.~Belardo, S.~M.~Cioab\u{a}, J.~H.~Koolen, and J.~Wang,
\emph{Open problems in the spectral theory of signed graphs},
Art Discrete Appl. Math. \textbf{1} (2018), no.~2, Paper No.~P2.10, 23 pp.

\bibitem{BermanShakedMonderer2003}
A.~Berman and N.~Shaked-Monderer,
\emph{Completely Positive Matrices},
World Scientific, River Edge, NJ, 2003.

\bibitem{CoutinhoSpier2023}
G.~Coutinho and T.~J.~Spier,
\emph{Sums of squares of eigenvalues and the vector chromatic number},
arXiv:2308.04475, 2023.

\bibitem{CoutinhoSpierZhang2024}
G.~Coutinho, T.~J.~Spier, and S.~Zhang,
\emph{Conic programming to understand sums of squares of eigenvalues of graphs},
arXiv:2411.08184, 2024.

\bibitem{ElphickAouchiche2017}
C.~Elphick and M.~Aouchiche,
\emph{Nordhaus--Gaddum and other bounds for the sum of squares of the positive eigenvalues of a graph},
Linear Algebra Appl. \textbf{530} (2017), 150--159.

\bibitem{ElphickFarberGoldbergWocjan2016}
C.~Elphick, M.~Farber, F.~Goldberg, and P.~Wocjan,
\emph{Conjectured bounds for the sum of squares of positive eigenvalues of a graph},
Discrete Math. \textbf{339} (2016), no.~9, 2215--2223.

\bibitem{ElphickLinz2024}
C.~Elphick and W.~Linz,
\emph{Symmetry and asymmetry between positive and negative square energies of graphs},
Electron. J. Linear Algebra \textbf{40} (2024), 418--432.

\bibitem{ElphickTangZhang2026}
C.~Elphick, Q.~Tang, and S.~Zhang,
\emph{A spectral lower bound on chromatic numbers using $p$-energy},
European J. Combin. \textbf{132} (2026), Part~B, Article No.~104252.

\bibitem{GerminaHameedZaslavsky2011}
K.~A.~Germina, S.~Hameed K., and T.~Zaslavsky,
\emph{On products and line graphs of signed graphs, their eigenvalues and energy},
Linear Algebra Appl. \textbf{435} (2011), no.~10, 2432--2450.

\bibitem{Harary1953}
F.~Harary,
\emph{On the notion of balance of a signed graph},
Michigan Math. J. \textbf{2} (1953/1954), no.~2, 143--146.

\bibitem{Hong1988}
Y.~Hong,
\emph{A bound on the spectral radius of graphs},
Linear Algebra Appl. \textbf{108} (1988), 135--139.

\bibitem{HornJohnson2013}
R.~A.~Horn and C.~R.~Johnson,
\emph{Matrix Analysis}, second ed.,
Cambridge University Press, Cambridge, 2013.

\bibitem{LiuNing2023}
L.~Liu and B.~Ning,
\emph{Unsolved problems in spectral graph theory},
Oper. Res. Trans. \textbf{27} (2023), no.~4, 33--60.

\bibitem{LiuTangZhang2026}
Y.~Liu, Q.~Tang, and S.~Zhang,
\emph{The positive and negative square-energy conjecture},
arXiv:2607.18031, 2026.

\bibitem{Stanic2026}
Z.~Stani\'c,
\emph{Spectra of Signed Graphs},
Cambridge University Press, Cambridge, 2026.

\bibitem{TangLiuWang2026}
Q.~Tang, Y.~Liu, and W.~Wang,
\emph{On the positive and negative $p$-energies of graphs under edge addition},
Discrete Appl. Math. \textbf{388} (2026), 25--33.

\bibitem{WocjanElphick2013}
P.~Wocjan and C.~Elphick,
\emph{New spectral bounds on the chromatic number encompassing all eigenvalues of the adjacency matrix},
Electron. J. Combin. \textbf{20} (2013), no.~3, Paper No.~P39.

\bibitem{Zaslavsky1982}
T.~Zaslavsky,
\emph{Signed graphs},
Discrete Appl. Math. \textbf{4} (1982), no.~1, 47--74;
erratum, ibid. \textbf{5} (1983), 248.

\bibitem{Zhang2024}
S.~Zhang,
\emph{Extremal values for the square energies of graphs},
arXiv:2409.15504, 2024.

\end{thebibliography}
\end{document}